\documentclass[10pt]{amsart}

\usepackage[colorlinks]{hyperref}
\usepackage{color,graphicx,shortvrb}
\usepackage[latin 1]{inputenc}

\usepackage[active]{srcltx} 

\usepackage{enumerate}
\usepackage{amssymb}
\usepackage{tikz-cd}

\usepackage [ all ]{xy}

\newtheorem{theorem}{Theorem}[section]

\newtheorem{corollary}[theorem]{Corollary}

\newtheorem{lemma}[theorem]{Lemma}

\newtheorem{proposition}[theorem]{Proposition}
\newtheorem{remark}[theorem]{Remark}

\def\J#1#2#3{ \left\{ #1,#2,#3 \right\} }

\def\NN{{\mathbb{N}}}
\def\11{\textbf{$1$}}
\def\CC{{\mathbb{C}}}

\DeclareGraphicsExtensions{.jpg,.pdf,.png,.eps}

\begin{document}

\title[Unitaries in JB$^*$-algebras]{Metric characterisation of unitaries in JB$^*$-algebras}

\author[M. Cueto-Avellaneda, A.M. Peralta]{Mar{\'i}a Cueto-Avellaneda, Antonio M. Peralta}

\address[M. Cueto-Avellaneda, A.M. Peralta]{Departamento de An{\'a}lisis Matem{\'a}tico, Facultad de
Ciencias, Universidad de Granada, 18071 Granada, Spain.}
\email{mcueto@ugr.es, aperalta@ugr.es}




\date{}

\begin{abstract} Let $M$ be a unital JB$^*$-algebra whose closed unit ball is denoted by $\mathcal{B}_M$. Let $\partial_e(\mathcal{B}_M)$ denote the set of all extreme points of $\mathcal{B}_M$. We prove that an element $u\in \partial_e(\mathcal{B}_M)$ is a unitary if and only if the set $$\mathcal{M}_{u} = \{e\in \partial_e(\mathcal{B}_M) : \|u\pm e\|\leq \sqrt{2} \}$$ contains an isolated point. This is a new geometric characterisation of unitaries in $M$ in terms of the set of extreme points of $\mathcal{B}_M$.
\end{abstract}

\maketitle
\thispagestyle{empty}

\section{Introduction}

We know from a celebrated result of R.V. Kadison that the extreme points of the closed unit ball of a C$^*$-algebra $A$ are
precisely the maximal partial isometries in $A$, that is, the elements $u$ in $A$ such that $(1 - uu^*)A(1 - u^*u) = \{0\}$ (see \cite{Kad51}). Every unitary in $A$ is an extreme point of its closed unit ball, but the reciprocal implication is not always true. In 2002, C.A. Akemann and N. Weaver searched 
for a characterisation of partial isometries, unitaries, and invertible elements in a unital C$^*$-algebra $A$ in terms of the Banach space structure of certain subsets of $A$, the dual space, $A^*$, or the predual, $A_*$, when $A$ is a von Neumann algebra (cf. \cite{AkWea2002}). The resulting characterisations are called geometric because only the Banach space structure of $A$ is employed. It should be noted that the geometric characterisation of partial isometries in a C$^*$-algebra was subsequently extended to a geometric characterisation of tripotents in a general JB$^*$-triple (see, \cite{FerMarPe2004,FerPe18}). The geometric characterisation of untaries actually relies on a good knowledge on the \emph{set of states} of a Banach space $X$ relative to an element $x$ in its unit sphere, $S(X)$, defined by $$S_x:=\{ \varphi\in X^* : \varphi (x) = \|\varphi\| =1\}.$$ The element $x$ is called a \emph{vertex} of the closed unit ball of $X$ (respectively, a \emph{geometric unitary} of $X$) if $S_x$ separates the points of $X$ (respectively, spans $X^*$).\smallskip

Akemann and Weaver proved that a norm-one element $x$ in a C$^*$-algebra $A$ is (an \emph{algebraic}) \emph{unitary} (i.e. $x x^* =x^* x =1$) if and only if $S_x$ spans $A^*$. In a von Neumann algebra $W$ an analogous characterisation holds when one uses the predual, $W_*$, in lieu of the dual space and the \emph{set of normal states relative to $x$}, $S^x=\{ \varphi\in W_* : \varphi (x) = \|\varphi\| =1\},$ in place of $S_x$ (cf. \cite[Theorem 3]{AkWea2002}).\smallskip

An appropriate versions of the just commented result in the setting of JB$^*$-algebras and JB$^*$-triples was established by A. Rodr{\'i}guez Palacios in \cite{Rod2010} (see section \ref{sec:topologically isolated} for the missing notions). We recall that a complex (respectively,  real) \emph{Jordan algebra} $M$ is a (not-necessarily associative) algebra over the complex (respectively, real) field whose product is abelian and satisfies $(a \circ b)\circ a^2 = a\circ (b \circ a^2)$ ($a,b\in M$). A \emph{normed Jordan algebra} is a Jordan algebra $M$ equipped with a norm, $\|.\|$, satisfying $\|
a\circ b\| \leq \|a\| \ \|b\|$ ($a,b\in M$). A \emph{Jordan Banach algebra} is a normed Jordan algebra whose norm is complete.
Every real or complex associative Banach algebra is a real Jordan Banach algebra with respect to the product $a\circ b: = \frac12 (a b +ba)$.\smallskip

An element $a$ in a unital Jordan Banach algebra $J$ is called invertible whenever there exists $b\in J$ satisfying $a \circ b = 1$ and $a^2 \circ b = a.$ The element $b$ is unique and it will be denoted by $a^{-1}$ (cf. \cite[3.2.9]{HOS}
).\smallskip

A \emph{JB$^*$-algebra} is a complex Jordan Banach algebra $M$ equipped
with an algebra involution $^*$ satisfying  $\|\J a{a}a \|= \|a\|^3$, $a\in
M$ (we recall that $\J a{a}a  = 2 (a\circ a^*) \circ a - a^2 \circ a^*$).
A \emph{JB-algebra} is a real Jordan Banach algebra $J$ in which the norm satisfies
the following two axioms for all $a, b\in J$\begin{enumerate}[$(i)$]\item $\|a^2\| =\|a\|^2$;
\item $\|a^2\|\leq \|a^2 +b^2\|.$
\end{enumerate} The hermitian part, $M_{sa}$, of a JB$^*$-algebra, $M$, is always a JB-algebra. A celebrated theorem due to J.D.M. Wright asserts that, conversely, the complexification of every JB-algebra is a JB$^*$-algebra (see \cite{Wri77}). We refer to the monographs \cite{HOS} and \cite{Cabrera-Rodriguez-vol1} for the basic notions and results in the theory of JB- and JB$^*$-algebras.\smallskip

Every C$^*$-algebra $A$ is a JB$^*$-algebra when equipped with its natural Jordan product $a\circ b =\frac12 (a b +b a)$ and the original norm and involution. Norm-closed Jordan $^*$-subalgebras of C$^*$-algebras are called \emph{JC$^*$-algebras}. \smallskip

Two elements $a, b$ in a Jordan algebra $M$ are said to \emph{operator commute} if $$(a\circ c)\circ b= a\circ (c\circ b),$$ for all $c\in M$. By the \emph{centre} of $M$ 
we mean the set of all elements of $M$ which operator commute with any other element in $M$. \smallskip

We recall that an element $u$ in a unital JB$^*$-algebra $M$ is a \emph{unitary} if it is invertible and its inverse coincides with $u^*$. An element $s$ in a unital JB-algebra $J$ is called a \emph{symmetry} if $s^2 =1$. The set of all symmetries in $J$ will be denoted by Symm$(J)$. If $M$ is a JB$^*$-algebra, we shall write Symm$(M)$ for Symm$(M_{sa})$.\smallskip

The geometric characterisation of unitaries in JB$^*$-algebras reads as follows: For a norm-one element $u$ in a JB$^*$-algebra $M$, the following conditions are equivalent:\begin{enumerate}[$(1)$]\item $u$ is a unitary in $M$;
\item $u$ is a geometric unitary in $M$;
\item $u$ is a vertex of the closed unit ball of $M$,
\end{enumerate} (see \cite[Theorem 3.1]{Rod2010} and \cite[Theorem 4.2.24]{Cabrera-Rodriguez-vol1}, where the result is proved in the more general setting of JB$^*$-triples).\smallskip

Surprisingly, as shown by C.-W. Leung, C.-K. Ng, N.-C. Wong in \cite{LeNgWon2009}, the case of JB-algebras differs slightly from the result stated for JB$^*$-algebras. Suppose $x$ is a norm-one element in a JB-algebra $J$, then the following statements are equivalent:\begin{enumerate}[$(a)$]\item $x$ is a geometric unitary in $J$;
\item $x$ is a vertex of the closed unit ball of $J$;
\item $x$ is an isolated point of Symm$(J)$ (endowed with the norm topology);
\item $x$ is a central unitary in $J$;
\item The multiplication operator $M_x: z\mapsto x\circ z$ satisfies $M_x^2 = \hbox{id}_{J}$,
\end{enumerate} (see \cite[Theorem 2.6]{LeNgWon2009} or \cite[Proposition 3.1.15]{Cabrera-Rodriguez-vol1}).\smallskip

Except perhaps statement $(c)$ above, the previous characterisations rely on the set of states $S_x$ of the underlying Banach space at an element $x$ in the unit sphere, that is, they are geometric characterisations in which the structure of the whole dual space plays an important role.\smallskip

From a completely independent setting, the different attempts to solve the problem of extending a surjective isometry between the unit spheres of two Banach spaces to a surjective real linear isometry between the spaces (known as Tingley's problem) have produced a substantial collection of new ideas and devices which are, in most of cases, interesting by themselves (cf., for example, \cite{BeCuFerPe2018,CabSan19,Mori2017,MoriOza2018,Pe2018}). Let us borrow some words from \cite{CabSan19} \emph{``...it is really impressive the development of machinery and technics that this problem {\rm(}Tingley's problem{\rm)} has led to.''}. We shall place our focus on the next result, included by M. Mori in \cite{Mori2017}, which provides a new characterisation of unitaries in a unital C$^*$-algebra.\smallskip

From now on, the closed unit ball of a Banach space $X$ will be denoted by $\mathcal{B}_X$. The set of all extreme points of a convex set $C$ will be denoted by $\partial_e(C)$.

\begin{theorem}\label{t metric charact of unitaries Mori}\cite[Lemma 3.1]{Mori2017} Let $A$ be a unital C$^*$-algebra, and let $u\in \partial_e(\mathcal{B}_A).$ Then the following statements are equivalent:
\begin{enumerate}[$(a)$]\item $u$ is a unitary {\rm(}i.e., $uu^* = u^*u=1${\rm)};
\item The set $\mathcal{A}_{u} = \{e\in \partial_e(\mathcal{B}_A) : \|u\pm e\|= \sqrt{2} \}$ contains an isolated point.
\end{enumerate}
\end{theorem}

The advantage of the previous result is that it characterises unitaries among extreme points of the closed unit ball of a unital C$^*$-algebra $A$ in terms of the subset of all points in $\partial_e(\mathcal{B}_A)$ at distance $\sqrt{2}$ from the element under study. We do not need to deal with the dual of $A$.\smallskip

The purpose of this note is to explore the validity of this characterisation in the setting of JB$^*$-algebras. In a first result we prove that for each tripotent $u$ in a JB$^*$-triple $E$ the equality $$\{ e\in \hbox{Trip}(E_2(u)) : \|u\pm e\| \leq \sqrt{2} \} = \{i(p-q) : p,q\in \mathcal{P}( E_2(u)) \hbox{ with }p\perp q\}$$ holds true, where given a JB$^*$-triple $E$, the symbol $\hbox{Trip}(E)$ stands for the set of all tripotents in $E$. Furthermore, if $u$ is unitary in $E$, then $$ \mathcal{E}_u =\left\{ e\in \partial_e(\mathcal{B}_E) : \|u\pm e\| \leq \sqrt{2} \right\}= i \hbox{Symm} (E_2(u))$$
$$=\{i (p-q): p,q\in \hbox{Trip}(E), p,q\leq u, p\perp q, p+q =u  \}$$ and the elements $\pm i u$ are isolated in $\mathcal{E}_u$ (Corollary \ref{c distance smaller than or equal to sqrt2}).\smallskip

After some technical results inspired from recent achievements by J. Hamhalter, O. F.~K. Kalenda, H. Pfitzner, and the second author of this note in \cite{hamhalter2019mwnc}, we arrive to our main result in Theorem \ref{t metric charact of unitaries JBalgebras}, where we prove the following: Let $u$ be an extreme point of the closed unit ball of a unital JB$^*$-algebra $M$. Then the following statements are equivalent:
\begin{enumerate}[$(a)$]\item $u$ is a unitary tripotent;
\item The set $\mathcal{M}_{u} = \{e\in \partial_e(\mathcal{B}_M) : \|u\pm e\|\leq \sqrt{2} \}$ contains an isolated point.
\end{enumerate}

\section{Background on JB$^*$-algebras and JB$^*$-triples}\label{sec:topologically isolated}

Suppose $A$ is a unital C$^*$-algebra whose set of projections (i.e. symmetric idempotents) will be denoted by $\mathcal{P} (A)$. It is known that the distance from $1$ to any projection in $\mathcal{P} (A)\backslash \{1\}$ is $1$, that is, $\|1-q\|\in \{0,1\}$ for all $q\in \mathcal{P} (A)$. Suppose $p$ is a central projection in $A$. In this case, $A$ writes as the orthogonal sum of $ p A p$ and $(1-p) A (1-p)$, and every projection $q$ in $A$ is of the form $q = q_1 +q_2$, where $q_1\leq p$ and $q_2\leq 1-p$. Then it easily follows that $\| p - q \| = \max \{ \|p-q_1\|,\|q_2\|\}\in \{0,1\}$ for each $q\in \mathcal{P} (A)$, which shows that $p$ is isolated (in the norm topology) in $\mathcal{P}(A)$. An easy example of a non-isolated projection can be given with 2 by 2 matrices. It is known that every rank one projection in $M_2(\mathbb{C})$ can be written in the form $p = \left(
            \begin{array}{cc}
              t & \gamma\sqrt{t(1-t)} \\
              \overline{\gamma}\sqrt{t(1-t)} & 1-t \\
            \end{array}
          \right)$, where $\gamma\in \mathbb{C}$ with $|\gamma|=1$ and $t\in [0,1]$. The mapping $q:[0,1]\to \mathcal{P}(M_2(\mathbb{C}))$, $q(s) = \left(
            \begin{array}{cc}
              s & \gamma\sqrt{s(1-s)} \\
              \overline{\gamma}\sqrt{s(1-s)} & 1-s \\
            \end{array}
          \right)$ is continuous and shows that $p$ is non-isolated in $\mathcal{P}(M_2(\mathbb{C}))$.
The natural question is whether $p$ being isolated in $\mathcal{P} (A)$ implies that $p$ is central in $A$. This question has been explicitly treated by M. Mori in \cite[Proof of Lemma 3.1]{Mori2017}. The argument is as follows, suppose $p$ is isolated in $\mathcal{P} (A)$, for each $a=a^*$ in $A$, we consider the mapping $\omega: \mathbb{R}\to \mathcal{P} (A)$, $\omega (t) :=e^{it a} p e^{-it a},$ which is differentiable with $\omega (0) = p$. We deduce from the assumption on $p$ that $\omega$ must be constant, and thus taking derivative at $t=0$ we get $i a p - i p a =0$, which implies that $p$ commutes with every hermitian element in $A$. That is every isolated projection in $\mathcal{P} (A)$ is central in $A$. We gather this information in the next result. 

\begin{proposition}\label{p equivalences of being isolated} Let $p$ be a projection in a unital C$^*$-algebra $A$. Then the following statements are equivalent:
\begin{enumerate}[$(a)$]\item $p$ is {\rm(}norm{\rm)} isolated in $\mathcal{P}(A)$;
\item $p$ is a central projection in $A$;
\item $1-2 p$ is {\rm(}norm{\rm)} isolated in Symm$(A)$.
\end{enumerate}
\end{proposition}

\begin{proof} The implication $(a)\Rightarrow (b)$ is proved in \cite[Proof of Lemma 3.1]{Mori2017}, while $(b)\Rightarrow (a)$ has been commented before. Finally it is easy to see that a sequence $(q_n)\subseteq \mathcal{P}(A)\backslash \{p\}$ converges in norm to $p$ if and only if the sequence  $(1-2 q_n)\subseteq \hbox{Symm}(A)\backslash \{1-2p\}$ converges in norm to $1-2 p$.
\end{proof}

A Jordan version of Proposition \ref{p equivalences of being isolated} was considered by J.D.M. Wright and M.A. Youngson in \cite{WriYou78}. Before going into details, let us note that the lacking of associativity for the product of a JB$^*$-algebra makes invalid the arguments presented above, and specially the use of products of the form $e^{it a} p e^{-it a}$ is not always possible in the Jordan analogue of $(a)\Rightarrow (b)$.\smallskip

In our approach to the Jordan setting, JB$^*$-algebras and JB-algebras will be regarded as JB$^*$-triples and real JB$^*$-triples, respectively. According to the original definition, introduced by W. Kaup in \cite{Ka83}, a JB$^*$-triple is a complex Banach space $E$ equipped with a continuous triple product $\J ... : E\times E\times E \to E,$ $(a,b,c)\mapsto \{a,b,c\},$ which is bilinear and symmetric in $(a,c)$ and conjugate linear in $b$,
and satisfies the following axioms for all $a,b,x,y\in E$:
\begin{enumerate}[{\rm (a)}] \item $L(a,b) L(x,y) = L(x,y) L(a,b) + L(L(a,b)x,y)
	- L(x,L(b,a)y),$ where $L(a,b):E\to E$ is the operator defined by $L(a,b) x = \J abx;$
\item $L(a,a)$ is a hermitian operator with non-negative spectrum;
\item $\|\{a,a,a\}\| = \|a\|^3$.\end{enumerate}

Examples of JB$^*$-triples include all C$^*$-algebras and JB$^*$-algebras with triple products of the form \begin{equation}\label{eq product operators} \J xyz =\frac12 (x y^* z +z y^*
x),\end{equation}  and \begin{equation}\label{eq product jordan}\J xyz = (x\circ y^*) \circ z + (z\circ y^*)\circ x -
(x\circ z)\circ y^*, \end{equation} respectively. \smallskip

Given an element $x$ in a JB$^*$-triple $E$, we shall write $x^{[1]} := x$, $x^{[3]} := \J xxx$, and $x^{[2n+1]} := \J xx{x^{[2n-1]}},$ $(n\in \NN)$.\smallskip

Analogously, as real C$^*$-algebras are defined as real norm closed hermitian subalgebras of C$^*$-algebras (cf. \cite{Li2003}), a real closed subtriple of a JB$^*$-triple is called a \emph{real JB$^*$-triple} (see \cite{IsKaRo95}). Every JB$^*$-triple is a real JB$^*$-triple when it is regarded as a real Banach space. In particular every JB-algebra is a real JB$^*$-triple with the triple product defined in \eqref{eq product jordan} (see \cite{IsKaRo95}).\smallskip

An element $e$ in a real or complex JB$^*$-triple $E$ is said to be a \emph{tripotent} if $\J eee=e$. 
Each tripotent $e\in E$, determines a decomposition of $X,$ known as the \emph{Peirce decomposition}\label{eq Peirce decomposition} associated with $e$, in the form $$E= E_{2} (e) \oplus E_{1} (e) \oplus E_0 (e),$$ where $E_j (e)=\{ x\in E : \J eex = \frac{j}{2}x \}$
for each $j=0,1,2$.\smallskip

Triple products among elements in the Peirce subspaces satisfy the following \emph{Peirce arithmetic}: $\J {E_{i}(e)}{E_{j} (e)}{E_{k} (e)}\subseteq E_{i-j+k} (e)$ if $i-j+k \in \{ 0,1,2\},$ and $\J {E_{i}(e)}{E_{j} (e)}{E_{k} (e)}=\{0\}$ otherwise, and $$\J {E_{2} (e)}{E_{0}(e)}{E} = \J {E_{0} (e)}{E_{2}(e)}{E} =0.$$ Consequently, each Peirce subspace $E_j(e)$ is a real or complex JB$^*$-subtriple of $E$.\smallskip

The projection $P_{k_{}}(e)$ of $E$ onto $E_{k} (e)$ is called the \emph{Peirce $k$-projection}. It is known that Peirce projections are contractive (cf. \cite[Corollary 1.2]{FriRu85}) and determined by the following identities $P_{2}(e) = Q(e)^2,$ $P_{1}(e) =2(L(e,e)-Q(e)^2),$ and $P_{0}(e) =\hbox{Id}_E - 2 L(e,e) + Q(e)^2,$ where $Q(e):E\to E$ is the conjugate or real linear map defined by $Q(e) (x) =\{e,x,e\}$. A tripotent $e$ in $E$ is called \emph{unitary} (respectively, \emph{complete} or \emph{maximal}) if $E_2 (e) = E$ (respectively, $E_0 (e) =\{0\}$). This definition produces no contradiction because unitary elements in a unital JB$^*$-algebra are precisely the unitary tripotents in $M$ when the latter is regarded as a JB$^*$-triple (cf. \cite[Proposition 4.3]{BraKaUp78}). A tripotent $e$ in $X$ is called \emph{minimal} if $E_2(e)=\CC e \neq \{0\}$. 
The set of all tripotents (respectively, of all complete tripotents) in a JB$^*$-triple $E$ will be denoted by $\hbox{Trip}(E)$ (respectively, $\hbox{Trip}_{max}(E)$).\smallskip

It is worth remarking that if $E$ is a complex JB$^*$-triple, the Peirce 2-subspace $E_2 (e)$ is a unital JB$^*$-algebra with unit $e$,
product $x\circ_e y := \J xey$ and involution $x^{*_e} := \J exe$, respectively.\label{eq Peirce-2 is a JB-algebra} 
\smallskip

Let us recall that a couple of elements $a,b$ in a real or complex JB$^*$-triple $E$ are called \emph{orthogonal} (written $a\perp b$) if $L(a,b) =0$. It is known that $a\perp b$ $\Leftrightarrow$ $\J aab =0$ $\Leftrightarrow$ $\{b,b,a\}=0$ $\Leftrightarrow$ $b\perp a$. If $e$ is a tripotent in $E$, it follows from Peirce rules that $a\perp b$ for every $a\in E_2(e)$ and every $b\in E_0(e)$.
Two projections $p,q$ in a JB$^*$-algebra are orthogonal if and only if $ p\circ q = 0$. An additional geometric property of orthogonal elements shows that $\|a\pm b\| = \max\{\|a\|,\|b\|\}$ whenever $a$ and $b$ are orthogonal elements in a real or complex JB$^*$-triple (cf. \cite[Lemma 1.3]{FriRu85}).\smallskip

Henceforth the set, $\hbox{Trip} (E)$, of all tripotents in a JB$^*$-triple $E$, will be equipped with the natural partial order\label{eq partial order tripotents} defined by $u \leq e$ in $\hbox{Trip} (E)$ if $e-u$ is a tripotent in $E$ with $e-u \perp u$, equivalently, if $u$ is a projection in the JB$^*$-algebra $E_2(e)$.\smallskip

One of the useful geometric properties of a real or complex JB$^*$-triple, $E$, asserts that the extreme points of its closed unit ball, $\mathcal{B}_E,$ are precisely the complete tripotents in $E$, that is, \begin{equation}\label{eq extrem points complete tripotents} \partial_e(\mathcal{B}_E) = \hbox{Trip}_{max}(E)
\end{equation} (cf. \cite[Proposition 3.5]{KaUp77} and \cite[Lemma 3.3]{IsKaRo95}). 
\smallskip

Let $a$ be a hermitian element in a JB$^*$-algebra $M$, the spectral theorem \cite[Theorem 3.2.4]{HOS} assures that the JB$^*$-subalgebra of $M$ generated by $a$ is isometrically JB$^*$-isomorphic to a commutative C$^*$-algebra. In particular, we can write $a$ as the difference of two orthogonal positive elements in $M_{sa}$. By applying this result it can be seen that every tripotent in $M_{sa}$ is the difference of two orthogonal projections in $M$, and furthermore, when $M$ is unital we obtain \begin{equation}\label{eq complete tripotents are symmetries} \partial_e (\mathcal{B}_{M_{sa}}) =\hbox{Symm} (M) = \{s\in M_{sa} : s^2 = 1 \}
\end{equation} (cf. \cite{WriYou78} or \cite[Proposition 3.1.9]{Cabrera-Rodriguez-vol1}). As in the associative case, the symbol $\mathcal{P}(M)$ will stand for the set of all projections (i.e., self-adjoint idempotents) in a JB$^*$-algebra $M$.\smallskip

The next result, which is a Jordan version of Proposition \ref{p equivalences of being isolated}, was originally established in \cite[Proposition 1.3]{IsRo95}, and a new proof can be consulted in \cite[Proposition 3.1.24 and Remark 3.1.25]{Cabrera-Rodriguez-vol1}. An alternative proof, based on the structure of real JB$^*$-triples, is included here for the sake of completeness.

\begin{proposition}\label{p equivalences of being isolated JB*-algebras}\cite[Proposition 1.3]{IsRo95}, \cite[Proposition 3.1.24]{Cabrera-Rodriguez-vol1}
Let $p$ be a projection in a unital JB$^*$-algebra $M$.
Then the following statements are equivalent:
\begin{enumerate}[$(a)$]\item $p$ is {\rm(}norm{\rm)} isolated in $\mathcal{P}(M)$;
\item $p$ is a central projection;
\item $1-2 p$ is {\rm(}norm{\rm)} isolated in Symm$(M)$.
\end{enumerate}
\end{proposition}

\begin{proof} The equivalence $(c)\Leftrightarrow (a)$ follows by the same arguments employed in the case of C$^*$-algebras.\smallskip

$(c)\Rightarrow (b)$ Suppose $1-2 p$ is (norm) isolated in Symm$(M)$. We consider $M_{sa}$ as a real JB$^*$-triple. Given $a,b\in M_{sa}$, by the axioms in the definition of JB$^*$-triples,  the mapping $$\Phi_t^{a,b} =\exp(t (L(a,b)-L(b,a)) = \sum_{n=0}^{\infty} \frac{t^n}{n!}(L(a,b)-L(b,a))^ n : M\to M$$ is a surjective linear isometry for all $t\in \mathbb{R}$, and clearly maps $M_{sa}$ into itself. One of the starring results in the theory of JB$^*$-triples asserts that every surjective linear isometry between JB$^*$-triples is a triple isomorphism (cf. \cite[Proposition 5.5]{Ka83}). Therefore $\Phi_t^{a,b}$ and $\Phi_t^{a,b}|_{M_{sa}} : M_{sa}\to M_{sa}$ are (isometric) triple isomorphisms. Since $1-2 p$ is an extreme point of the closed unit ball of $M_{sa}$, we deduce that $\Phi_t^{a,b}(1-2p)$ must be an extreme point of the closed unit ball of $M_{sa}$, and hence a complete tripotent in $M_{sa}$, or equivalently, a symmetry in $M$. Therefore the mapping $\omega :\mathbb{R}\to \hbox{Symm}(M)$, $t\mapsto \omega(t) =\Phi_t^{a,b}(1-2p)$ is differentiable with $\omega(0) = 1-2 p$. Since $1-2p$ is isolated in Symm$(M)$, the mapping $\omega(t) $ must be constant, and thus, by taking derivative at $t=0$ we get $$ 0 = (L(a,b)-L(b,a)) (1-2 p) = \{a,b,1-2 p\} - \{b,a,1-2p \},$$ equivalently, $$((1-2p)\circ a)\circ b = ((1-2p)\circ b)\circ a,$$ for all $a,b\in M_{sa}$ (and for all $a,b\in M$). This shows that $1-2p$ (and hence $p$) lies in the center of $M$ as desired.\smallskip

$(b)\Rightarrow (a)$ If $p$ is a central projection in $M$, we know from \cite[Lemma 2.5.5]{HOS} that $M= U_p(M) \oplus U_{1-p} (M)$, where for each $z\in M$, $U_z(x) = \{z,x^*,z\} = 2 (z\circ x) \circ z - z^2 \circ x$ ($\forall x\in M$). We further know that every element in $U_p(M)$ is orthogonal to every element in $U_{1-p} (M)$. Arguing as in the associative case (see Proposition \ref{p equivalences of being isolated} above), we prove that for each projection $q$ in $M$ we have $\|p-q\|\in \{0,1\}$, which concludes the proof.
\end{proof}

\section{Metric characterisation of unitaries}

Let us revisit some of the arguments in the proof of \cite[Lemma 3.1]{Mori2017} under the point of view of Jordan algebras.

\begin{proposition}\label{p MO for Cstar algebras from a Jordan point of view} Let $e$ be a maximal partial isometry in a unital C$^*$-algebra $A$, and let $l= ee^*$ and $r=e^* e$ denote the left and right projections of $e$. Suppose we can find two orthogonal projections $p,q\in A$ such that  $l = p + q$. Then the element $y = i (p-q)e$ lies in $\mathcal{A}_{e}= \left\{y\in \partial_e(\mathcal{B}_A) : \|e\pm y\|= \sqrt{2} \right\}$, and for each $\theta\in \mathbb{R}$ the element $$y_{\theta} := P_2(e^*) (y) + \cos(\theta) P_1 (e^*) (y) + \sin(\theta) P_1(e^*) (1)$$ is a maximal partial isometry in $A$.\smallskip

If we further assume that $p$ and $q$ are central projections in $l A l$, the following statements hold:
\begin{enumerate}[$(a)$]\item The elements $ p^\prime = e p e^*$ and $q^\prime = e q e^*$ are two orthogonal central projections in $r A r$, with $r= p^\prime + q^\prime$;
\item Suppose that $e$ is not unitary in $A$, and take $y = i (p-q)e$. Then $y$ lies in $\mathcal{A}_{e}$, and for each $\theta\in \mathbb{R}$ the element $y_{\theta} := P_2(e^*) (y) + \cos(\theta) P_1 (e^*) (y) + \sin(\theta) P_1(e^*) (1)$ is a maximal partial isometry in $A$ with  $\|e\pm y_{\theta}\| = \sqrt{2}$ {\rm(}actually, $\frac{e\pm y_{\theta}}{\sqrt{2}}$ is a maximal partial isometry in $A${\rm)}, and $y_{\theta}\neq y$ for all $\theta$ in $\mathbb{R}\backslash \left(2 \pi \mathbb{Z}\cup \pi \frac{1+2 \mathbb{Z}}{2}\right)$. Furthermore, $\|y-P_2(y)(y_{\theta})\|\leq 1-\cos(\theta)$, and hence $P_2(y)(y_{\theta})$ is invertible in $A_2(y)$ for $\theta$ close to zero.
\end{enumerate}
\end{proposition}

\begin{proof} Let us prove the first statement. Clearly, $y = i (p-q)e$ lies in $\mathcal{A}_{e}$. By \cite[Lemma 6.1]{hamhalter2019mwnc} there exist a complex Hilbert space $H$ and an isometric unital Jordan $*$-monomorphism $\Psi:A\to B(H)$ such that $\Psi(e)^*\Psi(e)=1$. Let us denote $v = \Psi(e)$, $z = \Psi(y)$, and $z_{\theta} = \Psi(y_{\theta})$. We observe that $$z_{\theta} = P_2(v^*) (z) + \cos(\theta) P_1 (v^*) (z) + \sin(\theta) P_1(v^*) (1),$$ because $\Psi$ is a unital Jordan $^*$-monomorphism, and hence it preserves triple products and involution. Clearly, $v= \Psi (e)$ is a maximal partial isometry (actually, an isometry $v^* v =1$) in $B(H)$. We shall write $B$ for $B(H)$. Having the above properties in mind we can rewrite $ z_{\theta}$ in the form
$$ z_{\theta} = v^*v z v v^* + \cos(\theta) \left( (1-v^*v) z v v^*  + v^*v z (1-v v^*) \right) $$ $$+ \sin(\theta) \left( (1-v^*v) 1 v v^*  + v^*v 1 (1-v v^*) \right)$$
$$ =  z v v^* + \cos(\theta) z (1-v v^*) + \sin(\theta) (1-v v^*).$$ Let us observe that the latter expression already appears in the proof of \cite[Lemma 3.1]{Mori2017}. 
\smallskip

Let us examine the element $z_{\theta}$ more closely. It follows from the properties commented above that
$$\begin{aligned}  z_{\theta}^* z_{\theta} &= v v^* z^* z v v^* + \cos(\theta) v v^* z^*  z (1-v v^*) + \sin(\theta) v v^* z^*  (1-v v^*) \\
&+ \cos(\theta)  (1-v v^*) z^* z v v^* + \cos^2(\theta)  (1-v v^*) z^*  z (1-v v^*) \\ &+ \cos(\theta)  \sin(\theta)  (1-v v^*) z^*  (1-v v^*) +\sin(\theta) (1-v v^*) z v v^* \\
&+ \sin(\theta)  \cos(\theta) (1-v v^*) z (1-v v^*) + \sin^2(\theta) (1-v v^*)\\
&= v v^* + \cos^2(\theta)  (1-v v^*) + \sin^2(\theta) (1-v v^*) = vv^* =1,
\end{aligned}$$ witnessing that $z_{\theta}$ is an isometry in $B$. It then follows from the properties of $\Psi$ that $y_{\theta} = \Psi^{-1} \left(\Psi(y_{\theta})\right)\in \partial_{e} (\mathcal{B}_{A})$ is a complete tripotent in $A$.\smallskip

Concerning the second statement, let us analyze the element $w=e\pm y_{\theta}$. As before, up to an application of \cite[Lemma 6.1]{hamhalter2019mwnc}, we can suppose that $r=e^* e =1$. We set $l = ee^*$. Assuming that $e$ is not unitary the projection $1-l=1-ee^*$ is not zero. We therefore have $$ w= e\pm y_{\theta} = (e\pm y ) l  + (e\pm \cos(\theta) y) (1-l) +\sin(\theta) (1-l),$$ and we shall compute $w^* w$.\smallskip

$(a)$ Let us make some observations. The mappings $\Phi_1: l A l\to r A r$,  $x\mapsto e^* x e$ and $\Phi_2 :r A r\to l A l$, $y\mapsto e y e^*$ are well defined, linear, and contractive. It is easy to see that $x = l x l = e (e^* x e) e^*= \Phi_2 \Phi_1 (x)$ and $y = e^* (e y e^*) e=  \Phi_1 \Phi_2 (y)$, for all $x\in l A l$ and $y\in r A r$. Therefore  $\Phi_2$ and $\Phi_1$ are linear bijections and inverses each other. Furthermore, for all $x,z\in l A l$, we have $ \Phi_1 (x)  \Phi_1 (z)=  (e^* x e) (e^* z e^*) = e (x z) e^* =  \Phi_1 (x z)$, and $\Phi_1 (x)^*= (e^* x e)^* = e^* x^* e=  \Phi_1 (x^*)$, for all $x\in l A l$, which shows that the first mapping is a unital C$^*$-isomorphism. Then the elements $ p^\prime = \Phi_1 (p)$ and $q^\prime = \Phi_1 (q)$ are two orthogonal central projections in $r A r=A$ with $1=r= p^\prime + q^\prime$.\smallskip

$(b)$ We derive from the above that $p e = e p^\prime,$ and $q e = e q^{\prime},$ essentially because $pe \perp eq^{\prime}$ and $qe\perp ep^{\prime}$. Consequently, $$y = i (p-q) e= i e (p^\prime -q^\prime), \ e\pm y = e ( \mu_{\pm} p^{\prime} + \overline{\mu_{\pm}} q^{\prime} ),$$ $$\hbox{ and } e\pm \cos(\theta) y = e\left(\lambda_{\pm} p^{\prime} + \overline{\lambda_{\pm}} q^{\prime}\right),$$ where $\mu_{\pm} =1\pm i$, and $\lambda_{\pm}=1\pm i \cos(\theta)$. We study next all summands involved in the product $w^* w$:

$$\begin{aligned}((x\pm y)l )^* ((x\pm y)l ) &= l (x\pm y)^* (x\pm y) l  =l  ( \overline{\mu_{\pm}} p^{\prime} + {\mu_{\pm}} q^{\prime} ) e^* e ( \mu_{\pm} p^{\prime} + \overline{\mu_{\pm}} q^{\prime} )\\
&= 2 l (p^{\prime} +q^{\prime}) l = 2 l ;\\
\sin(\theta) ((x\pm y)l )^* (1-l) &= \sin(\theta)  l  ( \overline{\mu_{\pm}} p^{\prime} + {\mu_{\pm}} q^{\prime} ) e^*  (1-l) =0;
\end{aligned}$$
$$\begin{aligned}((x\pm y)l )^* (e\pm \cos(\theta) y) (1-l)&= l  ( \overline{\mu_{\pm}} p^{\prime} + {\mu_{\pm}} q^{\prime} ) e^* e \left(\lambda_{\pm} p^{\prime} + \overline{\lambda_{\pm}} q^{\prime}\right) (1-l) \\
&= l  ( \lambda_{\pm} \overline{\mu_{\pm}} p^{\prime} + \overline{\lambda_{\pm}} {\mu_{\pm}} q^{\prime} )  (1-l);\\
(1-l)  (e\pm \cos(\theta) y)^* (x\pm y) l &= (1-l)  \left(\overline{\lambda_{\pm}} p^{\prime} + {\lambda_{\pm}} q^{\prime}\right) e^* e ( \mu_{\pm} p^{\prime} + \overline{\mu_{\pm}} q^{\prime} ) l\\
&= (1-l)  \left(\overline{\lambda_{\pm}} \mu_{\pm} p^{\prime} + {\lambda_{\pm}} \overline{\mu_{\pm}} q^{\prime}\right) l;
\end{aligned} $$
$$ (1-l)  (e\pm \cos(\theta) y)^* (e\pm \cos(\theta) y) (1-l) = (1-l)  \left(\overline{\lambda_{\pm}} p^{\prime} + {\lambda_{\pm}} q^{\prime}\right) e^* e \left(\lambda_{\pm} p^{\prime} + \overline{\lambda_{\pm}} q^{\prime}\right) (1-l) $$ $$=(1-l)  \left(\overline{\lambda_{\pm}} p^{\prime} + {\lambda_{\pm}} q^{\prime}\right)\left(\lambda_{\pm} p^{\prime} + \overline{\lambda_{\pm}} q^{\prime})\right) (1-l)= |\lambda_{\pm}|^2 (1-l)  \left( p^{\prime} + q^{\prime}\right) (1-l)$$ $$ =(1+\cos^2(\theta)) (1-l);$$
$$\begin{aligned}((e\pm \cos(\theta) y) (1-l))^* (\sin(\theta) (1-l))&=\sin(\theta) (1-l)  \left(\overline{\lambda_{\pm}} p^{\prime} + {\lambda_{\pm}} q^{\prime}\right) e^* (1-l)=0;\\
(\sin(\theta) (1-l))^* (x\pm y) l &= \sin(\theta) (1-l) e ( \mu_{\pm} p^{\prime} + \overline{\mu_{\pm}} q^{\prime} )l =0;\\
(\sin(\theta) (1-l))^* (e\pm \cos(\theta) y) (1-l) &=  \sin(\theta) (1-l) e \left(\lambda_{\pm} p^{\prime} + \overline{\lambda_{\pm}} q^{\prime}\right) (1-l) =0;\\
(\sin(\theta) (1-l))^* (\sin(\theta) (1-l)) &= \sin^2(\theta) (1-l) .
\end{aligned}$$

By adding the previous nine identities, and having in mind that $p^\prime$ and $q^\prime$ are central projections, we get
$$\frac{w^* w}{2} =  l + \frac12 (1+\cos^2(\theta)) (1-l) + \frac12 \sin^2(\theta) (1-l)$$ $$ + \frac12 l  ( \alpha p^{\prime} + \overline{\alpha} q^{\prime} )  (1-l) + \frac12 (1-l)  \left(\overline{\alpha} p^{\prime} + {\alpha} q^{\prime}\right) l = 1,$$ which proves that $\frac{w}{\sqrt{2}} = \frac{e\pm y_{\theta}}{\sqrt{2}}$ is an isometry, and consequently, $\|e\pm y_{\theta}\| = \sqrt{2}.$\smallskip

Let us now check that $y_{\theta}\neq y$ for all $\theta$ in $\mathbb{R}\backslash \left(2 \pi \mathbb{Z}\cup \pi \frac{1+2 \mathbb{Z}}{2}\right)$. Note that $l\neq 1$. Since $$\begin{aligned}(y-y_{\theta})^* (y-y_{\theta}) &= (1-\cos(\theta))^2 (1-l) y^*y (1-l) + \sin^2(\theta) (1-l) \\
& - (1-\cos(\theta)) \sin(\theta) (1-l) (y+y^*) (1-l)\\
& = 2 (1-\cos(\theta)) (1-l) - 2 (1-\cos(\theta)) \sin(\theta) a\\
&= 2 (1-\cos(\theta)) ((1-l)- \sin(\theta) a)\neq 0,
\end{aligned}$$ where $a=(1-l) \frac{y+y^*}{2} (1-l)$ is a hermitian element in the closed unit ball of $(1-l)A(1-l)$, and hence $\|\sin(\theta) a\|\leq |\sin(\theta)| <1$.\smallskip

Finally, the identity $$P_2(y) (y_{\theta}) = l y l r + \cos(\theta) l y (1-l) r +\sin(\theta) l (1-l) r = l y l + \cos(\theta) l y (1-l) $$ allows us to conclude that $\|y -P_2(y)(y_{\theta})\| = \|(1-\cos(\theta)) l y (1-l) \| \leq 1-\cos(\theta)$, which finishes the proof.
\end{proof}

Our goal in this section is to establish a similar characterisation of unitaries to that given in Theorem \ref{t metric charact of unitaries Mori} in the setting of JB$^*$-algebras and JB$^*$-triples.
It should be noted that the characterisation of unitaries in the case of JB$^*$-algebras is far from being a consequence of the result in the associative case.
We begin by describing the set of partial isometries at distance smaller than or equal to $\sqrt{2}$ from the unit of a JB$^*$-algebra. As observed by Mori in \cite{Mori2017}, in the easiest case $A =\mathbb{C}$, for $u\in\partial_e(\mathcal{B}_A) = \{ z\in\mathbb{C}: |z| =1\}$, we have $\mathcal{A}_u =\{ e\in \partial_e(\mathcal{B}_{\mathbb{C}}) : \|u\pm e\|= \sqrt{2} \}=\{iu, -i u\}$. But we can also add that $\mathcal{A}_u =\{ e\in \partial_e(\mathcal{B}_{\mathbb{C}}) : \|u\pm e\|\leq \sqrt{2} \}$.

\begin{lemma}\label{l distance sqrt2 between the unit and a tripotent} Let $M$ be a unital JB$^*$-algebra. Let $e$ be a tripotent in $M$ satisfying $\|1\pm e\|\leq \sqrt{2}$. Then there exist two orthogonal projections $p,q$ in $M$ such that $e = i (p-q)$. Consequently,
$$ \left\{ e\in \hbox{Trip}(M) : \|1\pm e\| \leq \sqrt{2} \right\} = \left\{i(p-q) : p,q\in\mathcal{P}(M) \hbox{ with }p\perp q\right\}.$$
\end{lemma}

\begin{proof} Let $N$ denote the JB$^*$-subalgebra of $M$ generated by $1,e$ and $e^*$. It follows from the Shirshov-Cohn theorem \cite[Theorems 2.4.14 and 7.2.5]{HOS}, combined with Wright's theorem \cite[Corollary 2.2 and subsequent comments]{Wri77}, that $N$ is special, that is, there exists a unital C$^*$-algebra $A$ containing $N$ as unital JB$^*$-subalgebra. The C$^*$-algebra $A$ contains $1$ and the partial isometry $e$ and we have $\|1\pm e\|\leq \sqrt{2}$. Let us write  $l=ee^*$ and $r= e^* e$ for the left an right projections of $e$ in $N$. Then, it follows that \begin{equation}\label{eq 1 0806} 0\leq\frac12( 1 +l \pm (e+e^*)) =\frac12  (1\pm e) (1\pm e)^* \leq \frac12 \|(1\pm e) (1\pm e)^*\| 1 = \frac{\|1\pm e\|^2}{2} 1 \leq 1,
 \end{equation}which implies that $2 l \pm U_l (e+e^*) \leq 2 l,$ where we have applied that the mapping $U_l$ is positive. Therefore $U_l (e+e^*) =0$, and it follows from the definition of $l$ that $U_{1-l} (e) = U_{1-l} (e^*)=0$. Back to \eqref{eq 1 0806} we get $$2 l + (1-l) \pm e (1-l) \pm (1-l) e^*  = 1 +l \pm (e+e^*) \leq 2 \ 1 = 2 l + 2 (1-l),$$ inequality which implies that $$ \pm (e (1-l) + (1-l) e^*)\leq 1-l $$ and hence $e (1-l) + (1-l) e^* =0,$ or equivalently, $e (1-l) =- (1-l) e^*$. We have shown that $$e+e^* = U_l (e+e^*)+ (1-l) (e+e^*) l + l (e+e^*) (1-l) +  U_{1-l} (e+e^*) = 0,$$ that is $e = -e^*$ is a skew symmetric partial isometry in $A$, and thus there exist two orthogonal projections $p,q$ in $A$ such that $e = i (p-q)$. Since $e= i (p-q)\in M,$ it follows that $e^2 = -p - q$ and $p-q$ both belong to $M$, and consequently, $p,q\in M,$  which concludes the proof.
\end{proof}

Given a tripotent $u$ in a JB$^*$-triple $E$, the Peirce 2-subspace $E_2(u)$ is a unital JB$^*$-algebra with unit $u$ (see page \pageref{eq Peirce-2 is a JB-algebra}). So, the first statement in the next corollary is a straight consequence of our previous lemma.

\begin{corollary}\label{c distance smaller than or equal to sqrt2} Let $u$ be a tripotent in a JB$^*$-triple $E$. Then
$$\{ e\in \hbox{Trip}(E_2(u)) : \|u\pm e\| \leq \sqrt{2} \} = \{i(p-q) : p,q\in \mathcal{P}( E_2(u)) \hbox{ with }p\perp q\}.$$  Furthermore, if $u$ is unitary in $E$, then
\begin{equation}\label{eq 2 0806} \mathcal{E}_u =\left\{ e\in \partial_e(\mathcal{B}_E) : \|u\pm e\| \leq \sqrt{2} \right\}= i \hbox{Symm} (E_2(u))
\end{equation} $$=\{i (p-q): p,q\in \hbox{Trip}(E), p,q\leq u, p\perp q, p+q =u  \}$$ and the elements $\pm i u$ are isolated in $\mathcal{E}_u$.
\end{corollary}

\begin{proof} The first statement is a consequence of Lemma \ref{l distance sqrt2 between the unit and a tripotent}. If $u$ is unitary the equality $E=E_2(u)$ holds. Having in mind that $\partial_e(\mathcal{B}_E) =\hbox{Trip}_{max} (E)$, we deduce from the first statement that $$\mathcal{E}_u \subseteq \{i(p-q) : p,q\in \hbox{Trip}(E), p,q\leq u, p\perp q\}.$$ But every $e= i(p-q)\in \mathcal{E}_u$ must be also a complete tripotent in $E$, which forces $p+q=u,$ otherwise $r=u-p-q$ would be a non-zero element in $E_0(e)$, which is impossible, so \eqref{eq 2 0806} is clear. It is obvious that $\pm i u\in \mathcal{E}_u$ and for any $i (p-q)\in \mathcal{E}_u\backslash \{\pm i u\}$ we have $$\| i u \pm i (p-q) \| = \| i (1\pm 1 ) p + i (1\mp 1) q\| =  \max\{\|(1\pm 1 ) p\|,\| (1\mp 1) q\| \}= 2.$$ This proves that $\pm i u$ are isolated in $\mathcal{E}_u$.
\end{proof}

The Jordan version of the Theorem \ref{t metric charact of unitaries Mori}$(a)\Rightarrow (b)$ has been established in Corollary \ref{c distance smaller than or equal to sqrt2} even in the setting of JB$^*$-triples. For the reciprocal implication we shall first prove a technical result which also holds for JB$^*$-triples.

\begin{proposition}\label{p metric charact of unitaries JBtriples partial sufficient} Let $u$ be a tripotent in a JB$^*$-triple $E$, and let $$\mathcal{E}_{u} = \{e\in \partial_e(\mathcal{B}_E) : \|u\pm e\|\leq \sqrt{2} \}.$$ Then every element $y\in \mathcal{E}_u$ with $P_1 (u) (y) \neq 0$ or $P_0(u) (y) \neq 0$ is non-isolated in $\mathcal{E}_{u}$. Consequently, every isolated element $y\in \mathcal{E}_{u}$ belongs to $i \hbox{Symm}(E_2(u))$.
\end{proposition}

\begin{proof}  Let us take $y \in \mathcal{E}_{u}$  with $P_1(u) (y) \neq 0$ or $P_0(u) (y) \neq 0$. By \cite[Lemma 1.1]{FriRu85} for each $\lambda\in \mathbb{C}$ with $|\lambda| = 1$ the mapping $S_{\overline{\lambda}} (u) = \overline{\lambda}^2 P_2(u) + \overline{\lambda} P_1(u) + P_0(u) =  \overline{\lambda}^2 P_2(u) + \overline{\lambda} P_1(u) + P_0(u) $ is an isometric triple isomorphism on $E$. Therefore the mapping
$R_{{\lambda}} (u) ={\lambda}^2 S_{\overline{\lambda}} (u) =  P_2(u) + {\lambda} P_1(u)+ \lambda^2 P_0(u) $ is an isometric triple isomorphism on $E$ for all $\lambda$ in the unit sphere of $\mathbb{C}$. Since Peirce projections are contractive $$\| y -  R_{{\lambda}} (u) (y) \|\geq \max\left\{ |\lambda-1| \| P_1(u) (y)\|, |\lambda^2-1| \| P_0(u) (y)\|  \right\} >0,$$ for all $\lambda\in \mathbb{T}\backslash\{\pm1 \}$. Clearly, $\displaystyle R_{{\lambda}} (u) (y) \stackrel{{\lambda\to 1}}{\longrightarrow} y$ in norm.\smallskip

On the other hand, $R_{\lambda} (u) (u) = u $ for all  $|\lambda|=1$. Since $R_{\lambda}(u)$ is an isometric triple automorphism on $E$ and $y\in\partial_e(\mathcal{B}_{E})$ we deduce that $R_{\lambda}(u) (y)\in \partial_e(\mathcal{B}_{E}),$ and $$\|u \pm R_{\lambda}(u) (y) \| = \| R_{\lambda}(u) (u) \pm R_{\lambda}(u) (y) \| =\| R_{\lambda}(u) (u\pm y)  \| = \|u\pm y\|\leq \sqrt{2},$$ for all  $|\lambda|=1$. Therefore $y$ is non-isolated in $\mathcal{E}_{u}$, which concludes the proof of the first statement.\smallskip

For the last statement, let us assume that $y\in \mathcal{E}_{u}$ is an isolated point. It follows from the first statement that $P_1(u) (y) = 0 = P_0(u) (y)$. That is, $y\in \partial_e(\mathcal{B}_{E})\cap E_2(u)$ with $\|u \pm y \|\leq \sqrt{2}$. We conclude from Corollary \ref{c distance smaller than or equal to sqrt2} that $y \in i \hbox{Symm} (E_2(u))$.
\end{proof}

\begin{remark}\label{r new july 10}{\rm The arguments given the proof of Proposition \ref{p metric charact of unitaries JBtriples partial sufficient} are valid to establish the following: Let $u$ be a tripotent in a JB$^*$-triple $E$, and let $$\displaystyle \tilde{\mathcal{E}}_{u} = \{e\in \hbox{Trip}(E) : \|u\pm e\|\leq \sqrt{2} \}.$$ Then every element $y\in \tilde{\mathcal{E}}_u$ with $P_1 (u) (y) \neq 0$ or $P_0(u) (y) \neq 0$ is non-isolated in $\tilde{\mathcal{E}}_{u}$.}
\end{remark}

We continue gathering the tools and results needed in our characterisation of unitaries in JB$^*$-algebras. One of the most successful tools in the theory of Jordan algebras is the Shirshov-Cohn theorem, which affirms that the JB$^*$-subalgebra of a JB$^*$-algebra generated by two symmetric elements (and the unit element) is a JC$^*$-algebra, that is, a JB$^*$-subalgebra of some $B(H)$ (cf. \cite[Theorem 7.2.5]{HOS} and \cite[Corollary 2.2]{Wri77}). The next lemma is an appropriate version of the Shirshov-Cohn theorem.

\begin{lemma}\label{l JB*-subalgebra generated by two orthogonal tripotents} Let $u_1$ and $u_2$ be two orthogonal tripotents in a unital JB$^*$-algebra $M$. Then the JB$^*$-subalgebra $N$ of $M$ generated by $u_1,u_1^*,u_2, u_2^*$ and the unit element is a JC$^*$-algebra, that is, there exists a complex Hilbert space $H$ satisfying that $N$ is a JB$^*$-subalgebra of $B(H)$, we can further assume that the unit of $N$ coincides with the identity on $H$.
\end{lemma}

\begin{proof} Let us fix $t\in (0,1)$. We consider the element $e=u_1+t u_2$. Let $N_0$ denote the JB$^*$-subalgebra of $M$ generated by $e$, $e^*$ and the unit element. It follows from the Shirshov-Cohn theorem that $N_0$ is a JC$^*$-algebra. We observe that $N_0$ is a JB$^*$-subtriple of $M$, therefore the element $e^{[2n-1]}$ belongs to $N_0$ for all natural $n$. Now, applying that $u_1$ and $u_2$ are two orthogonal tripotents, we can deduce that $$e^{[2n-1]} = u_1 + t^{(2n-1)} u_2 .$$ The sequence $(e^{[2n-1]})_n = (u_1 + t^{(2n-1)} u_2)_n$ converges in norm to $u_1$, and thus $u_1$ lies in $N_0$. Consequently, $u_1$ and $u_2$ both belong to $N_0$. \smallskip

Since $N_0$ and $N$ are JB$^*$-subalgebras of $M$, $u_1,u_2\in N_0$ and clearly $e\in N$, it follows from their definition that $N=N_0$ is a JC$^*$-algebra.\smallskip

The final statement can be obtained as in the proof of \cite[Lemma 6.2]{hamhalter2019mwnc}.
\end{proof}

The next result is inspired by \cite[Lemmata 6.2 and 6.3]{hamhalter2019mwnc}.

\begin{proposition}\label{p complete tripotents in the Cstar algebra generated} Let $u_1$ and $u_2$ be two orthogonal tripotents in a unital JB$^*$-algebra $M$ satisfying the following properties: \begin{enumerate}[$(a)$]\item $u= u_1+u_2$ a complete tripotent in $M$;
\item $u_1,u_2$ are central projections in the JB$^*$-algebra $M_2(u)$.
\end{enumerate} Let $N$ denote the JB$^*$-subalgebra of $M$ generated by $u_1$, $u_2$ and the unit element.
Then $N$ is a JC$^*$-subalgebra of some C$^*$-algebra $B$, and $u$ is a complete tripotent in the C$^*$-subalgebra $A$ of $B$ generated by $N$. Moreover, the elements $u_1,u_2$ are central projections in the JB$^*$-algebra $A_2(u)$.
\end{proposition}

\begin{proof} Lemma \ref{l JB*-subalgebra generated by two orthogonal tripotents} guarantees that $N$ is a JB$^*$-subalgebra of a unital C$^*$-algebra $B$, and we can also assume that $N$ contains the unit of $B$. Clearly, $u$, $u_1$ and $u_2$ are partial isometries in $A$. Let $l_i=u_i u_i^*$ and $r_i=u_i^* u_i$ denote the left and right projections of $u_i$ in $A$ ($i=1,2$). We shall also write $l = u u^*= l_1+l_2$ and $r=u^* u= r_1+r-2$, for the left and right projections of $u$ in $A$, respectively. Let us note that $l_1\perp l_2$ and $r_1\perp r_2$.\smallskip

By hypothesis, $u_1,u_2$ are central projections in the JB$^*$-algebra $M_2(u)$, and hence in $N_2(u)$. It then follows that the identity $$l N r = N_2(u)= N_2(u_1)\oplus^{\infty} N_2(u_2)= l_1 N r_1\oplus^{\infty} l_2 N r_2$$ holds. Having in mind that $1\in N$, we deduce that $lr = l_1 r_1 + l_2 r_2,$ and so $l_1 r = l_1 r_1$, which proves that $l_1 r_2 =0$. We can similarly prove that $l_2 r_1 =0$.\smallskip

Let $A$ denote the C$^*$-subalgebra of $B$ generated by $N$. We shall next show that $u$ is a complete tripotent in $A$. We know that $u$ is a complete tripotent in $M$, and hence in $N$. Clearly $u$ is a tripotent in $A$. The Peirce 0-projection on $A$ is given by $P_0 (u) (x) = (1-l) x (1-r)$ ($x\in A$). We therefore know that $(1-l) x (1-r) =0,$ for all $x\in N$. We shall prove that $(1-l) x (1-r) =0$ for all $x\in A$. For this purpose we shall adapt some technique from the proof of \cite[Lemma 6.2]{hamhalter2019mwnc}.\smallskip

Since $N$ is a JB$^*$-subalgebra of $A$, for each $n\in \mathbb{N}\cup\{0\}$, the elements $(u_1^*)^{n}$ and $(u_2^*)^{n}$ lie in $N$, and hence $(1-l) (u_1^*)^{n} (1-r)=(1-l) (u_2^*)^{n} (1-r)  =0$, or equivalently,
\begin{align}\label{eq powers on the kernel of Peirce 0} l_1 (u_1^*)^{n} (1-r) &=l (u_1^*)^{n} (1-r) = (u_1^*)^{n} (1-r), \hbox{ and } \\
l_2 (u_2^*)^{n} (1-r) &= l (u_2^*)^{n} (1-r) = (u_2^*)^{n} (1-r),\nonumber
\end{align} where in the first two equalities we applied that $l_1 r_2 =l_2 r_1 =0$.\smallskip

Fix $t\in (0,1)$. We have shown in the proof of Lemma \ref{l JB*-subalgebra generated by two orthogonal tripotents} that $N$ coincides with the JB$^*$-subalgebra of $M$ generated by $e=u_1 +t u_2$ and $1$. Let $A_0$ denote the set of all finite products of $e$ and $e^*$ and $1$. Since $A$ is the closed linear span of $A_0$ we only need to prove that $(1-l)x(1-r)=0$, for all $x\in A_0$.\smallskip

We say that an element $x\in A$ satisfies property $(\diamond)$ if \begin{equation*}\label{eq propety diamond}\hbox{$x(1-r)=0,$ or $x(1-r)=(1-r),$ or $x(1-r)=(u_1^*)^n(1-r)+ t^m (u_2^*)^n (1-r),$}
\end{equation*}  for some $n,m\in\mathbb{N}$.\smallskip

Let us fix an element $x\in A$ satisfying property $(\diamond)$. If $x(1-r)=0,$ we have $e^* x(1-r)=0,$ and $e x(1-r)=0.$ If $x(1-r)=(1-r),$ it follows that $$e^* x (1-r)= e^* (1-r) = u_1^* (1-r) + t e_2^* (1-r), \hbox{ and } e x(1-r) = e(1-r) =0.$$
If $x(1-r)=(u_1^*)^n (1-r)+ t^m (u_2^*)^n (1-r),$ for some $n,m\in\mathbb{N}$, it can be seen that $$e^* x(1-r)= e^* (u_1^*)^n (1-r)+ t^m e^*(u_2^*)^n (1-r) = (u_1^*)^{n+1} (1-r)+ t^{m+1} (u_2^*)^{n+1} (1-r),$$ where we applied that $u_1\perp u_2$, $l_1 r_2 =0,$ and $l_2 r_1 =0$. This shows that $e^* x$ satisfies property $(\diamond)$.\smallskip

In the latter case, by applying $u_1\perp u_2$, $l_1 r_2 =0,$ and $l_2 r_1 =0$, we also have $$\begin{aligned}e x(1-r) &=e (u_1^*)^n (1-r)+ t^m e (u_2^*)^n (1-r) \\
&= u_1 (u_1^*)^n (1-r)+ t^{m+1} u_2 (u_2^*)^n (1-r)\\
&= (u_1 u_1^*) (u_1^*)^{n-1} (1-r)+ t^{m+1} (u_2 u_2^*) (u_2^*)^{n-1} (1-r) \\
&= l_1  (u_1^*)^{n-1} (1-r)+ t^{m+1} l_2 (u_2^*)^{n-1} (1-r)\\
&=\hbox{(by \eqref{eq powers on the kernel of Peirce 0})} =(u_1^*)^{n-1} (1-r)+ t^{m+1} (u_2^*)^{n-1} (1-r),
\end{aligned} $$ witnessing that $e x$ satisfies property $(\diamond)$.\smallskip

We have proved that if $x$ satisfies property $(\diamond)$, then $ex$ and $e^*x$ both satisfy property $(\diamond)$. It is not hard to check that $1,e,$ and $e^*$ satisfy property $(\diamond)$. We can thus conclude that every element in $A_0$ satisfies property $(\diamond)$. So, for each $x\in A_0$ we have $(1-l) x (1-r) =0$ if $x (1-r) =0$. If $x (1-r) =(1-r)$, it follows from the fact that $1\in N$ and $u$ is complete in $N$, that $$(1-l) x (1-r) =(1-l)(1-r)= (1-l) 1 (1-r)= 0.$$ Finally, if $x(1-r)=(u_1^*)^n(1-r)+ t^m (u_2^*)^n (1-r),$ for some $n,m\in\mathbb{N}$, we easily check that $$ (1-l) x (1-r)= (1-l) (u_1^*)^n(1-r)+ t^m (1-l) (u_2^*)^n (1-r)=0,$$ where in the last equality we applied that $(u_1^*)^n, (u_2^*)^n\in N$ and $u$ is a complete tripotent in $N$. This proves that $(1-l) A_0 (1-r) =\{0\}$, and hence $u$ is complete in $A$.\smallskip

It remains to prove that $u_1$ and $u_2$ are central projections in $A_2(u)$. We claim that \begin{equation}\label{eq projections are central in the Cstar generated} l_1 A r_2 = l_2 A r_1 =\{0\}.
 \end{equation}Indeed, it is enough to prove that \begin{equation}\label{eq new induction statement} l_1 (x_1\cdots x_m) r_2 = l_2 (x_1\cdots x_m) r_1 =0,
 \end{equation} for all natural $m$ and $x_1,\ldots, x_m\in\{e,e^*\}$ because $N$ is the JB$^*$-subalgebra of $M$ generated by $e,e^*$ and the unit. We shall prove \eqref{eq new induction statement} by induction on $m$. We know from the hypotheses that $l_1 N r_2 = l_2 N r_1 =\{0\},$ so the case, $m=1$ is clear. \smallskip

The case $m=2$ is worth to be treated independently. The products of three elements are the following: $e^2, (e^*)^2, ee^*$ and $e^*e$. The elements $e^2$ and $(e^*)^2$ belong to $N$, and thus $l_1 e^2 r_2 = l_2 e^2 r_1 = l_1 (e^*)^2 r_2 = l_2 (e^*)^2 r_1 =0.$ By the properties seen in the above paragraphs we have $$ l_1 e e^* r_2= e r_1 e^* r_2 = e e^* l_1 r_2 =0. $$ Since $e\circ e^*\in N$, it follows that $l_1 (e e^* + e^* e) r_2=0.$ The last two equalities together give $$ l_1 e e^* r_2=  l_1 e^* e r_2 =0. $$ Similar arguments show that $$ l_2 e e^* r_1=  l_2 e^* e r_1 =0. $$

Suppose by the induction hypothesis that \eqref{eq new induction statement} for all natural numbers $2\leq m\leq m_0$. Let us make an observation, for any natural $k\leq m_0-1$ it follows from the induction hypothesis that $$l_1 (x_1\cdots x_{k}) l_2 e = l_1 x_1\cdots x_{k} e r_2 =0,$$ therefore
$$0 = (l_1 (x_1\cdots x_{k}) l_2 e)(l_1 (x_1\cdots x_{k}) l_2 e)^* = l_1 (x_1\cdots x_{k}) l_2 e e^* l_2 (x_k^*\cdots x_1^*) l_1 $$ $$= l_1 (x_1\cdots x_{k}) l_2  l_2 (x_k^*\cdots x_1^*) l_1 = \left(  l_1 (x_1\cdots x_{k}) l_2  \right) \left(  l_1 (x_1\cdots x_{k}) l_2  \right)^*,$$ witnessing that
\begin{equation}\label{eq sub induction hypothesis} l_1 (x_1\cdots x_{k}) l_2  = 0, \hbox{ for all natural } k\leq m_0-1.
\end{equation}

We deal next with the case $m_0+1$. We pick $x_1,\ldots x_{m_0},$ $x_{m_0+1}$ $\in $ $\{e,e^*\}$. Since $e^{m+1},(e^*)^{m+1}\in N$, the desired conclusion is clear for $x_1=\ldots=x_{m+1} =e$ and $x_1=\ldots=x_{m+1} =e^*$. We can therefore assume the existence of $j \in \{1,\ldots,m_0\}$ such that $x_j x_{j+1} = e^* e=1$ or $x_j x_{j+1} = e e^*$. In the first case $$l_1 x_1\cdots x_{m_0+1} r_2 = l_1 x_1\cdots x_{j-1} x_{j+1}\cdots x_{m_0+1} r_2 = 0,$$ by the induction hypothesis. In the second case we have $$l_1 x_1\cdots x_{m_0+1} r_2 = l_1 x_1\cdots x_{j-1} l x_{j+1} \cdots x_{m_0+1} r_2$$ $$=l_1 x_1\cdots x_{j-1} l_1 x_{j+1} \cdots x_{m_0+1} r_2 + l_1 x_1\cdots x_{j-1} l_2 x_{j+1} \cdots x_{m_0+1} r_2=0,$$ where in the last equality we applied \eqref{eq sub induction hypothesis} and the induction hypothesis.\smallskip

Similar ideas to those we gave above are also valid to establish  $$l_2 x_1\cdots x_{m} r_1 =0,\hbox{ for all } m\in \mathbb{N}, x_1\cdots x_{m}\in\{e,e^*\}.$$ This finishes the induction argument and the proof of the claim in \eqref{eq projections are central in the Cstar generated}. It follows from \eqref{eq projections are central in the Cstar generated} that $u_1$ and $u_2$ are central projections in $A_2(u)$.
\end{proof}

%
%
%
%
%

The desired characterisation of unitaries in a unital JB$^*$-algebra is now established in our main result.

\begin{theorem}\label{t metric charact of unitaries JBalgebras} Let $u$ be an extreme point of the closed unit ball of a unital JB$^*$-algebra $M$.
Then the following statements are equivalent:
\begin{enumerate}[$(a)$]\item $u$ is a unitary tripotent;
\item The set $\mathcal{M}_{u} = \{e\in \partial_e(\mathcal{B}_M) : \|u\pm e\|\leq \sqrt{2} \}$ contains an isolated point.
\end{enumerate}
\end{theorem}

\begin{proof} Corollary \ref{c distance smaller than or equal to sqrt2} gives $(a)\Rightarrow (b)$.\smallskip

$(b)\Rightarrow (a)$ We shall show that if $u$ is not a unitary tripotent then every point $y \in \mathcal{M}_{u}$ is non-isolated. We therefore assume that $u$ is not a unitary tripotent. Let us fix $y\in \mathcal{M}_{u}$. If $P_1(u)(y)\neq 0$, Proposition \ref{p metric charact of unitaries JBtriples partial sufficient} implies that $y$ is non-isolated in $\mathcal{M}_{u}$. We can therefore assume that $P_1(u)(y) =0$, and hence $y = P_2(u)(y)$. So, $y$ and $u$ lie in the JB$^*$-algebra $M_2(u)$ (we observe that the latter need not be a JB$^*$-subalgebra of $M$). Since $y$ also is an extreme point of the closed unit ball of $M_2(u)$ and $\|u \pm y \| \leq  \sqrt{2},$ Corollary \ref{c distance smaller than or equal to sqrt2} implies that $y$ lies in $i \hbox{Symm}(M_2(u))$, therefore, there exist orthogonal tripotents $u_1,u_2\in M$ with $u_1,u_2\leq u$, $u_1+u_2 = u$ and $y = i (u_1-u_2)$.\smallskip

If $u_2$ is non-isolated in $\mathcal{P}(M_2(u))$, then there exists a sequence $(q_n)_n\subseteq \mathcal{P}(M_2(u))$ with $q_n\neq u_2,$ for all $n$, converging to $u_2$ in norm. In this case the sequence $(i (u-2 q_n))_n$ is contained in $\mathcal{M}_u\backslash \{y= i (u_1-u_2)\}$ (let us observe that $u- 2 q_n$ is a symmetry in $M_2(u)$ and since $u\in \partial_e(\mathcal{B}_M)$, \cite[Lemma 4]{Sidd2007} implies that $i(u-2q_n)\in \partial_e(\mathcal{B}_M)$ for all $n\in \mathbb{N}$, and clearly $\|u \pm i(u-2q_n)\|= \sqrt{2}$) and converges to $y$ in norm. We have therefore shown that $y$ is non-isolated in $\mathcal{M}_{u}$.\smallskip

We finally assume that $u_2$ is isolated in $\mathcal{P}(M_2(u))$. In this case Proposition \ref{p equivalences of being isolated JB*-algebras} proves that $u_2$ (and hence $u_1$) is a central projection in $M_2(u)$. We are in position to apply Proposition \ref{p complete tripotents in the Cstar algebra generated} to the tripotents $u_1,$ $u_2$ and $u= u_1+u_2$ in $M$. Let $N$ denote the JB$^*$-subalgebra of $M$ generated by $u_1$, $u_2$ and the unit element. By the just quoted proposition, $N$ is a JC$^*$-subalgebra of some C$^*$-algebra $B$, $u$ is a complete tripotent in the C$^*$-subalgebra $A$ of $B$ generated by $N$, and the elements $u_1,u_2$ are central projections in the JB$^*$-algebra $A_2(u)$. Let us observe that $u$ and $y$ both belong to $N$ (and to $A$).  Proposition \ref{p MO for Cstar algebras from a Jordan point of view}, applied to $A$, $u$, $p=u_1 u_1^*$, $q = u_2 u_2^*$, and $y$, implies that for each $\theta\in \mathbb{R}$ the element $$y_{\theta} := P_2(u^*) (y) + \cos(\theta) P_1 (u^*) (y) + \sin(\theta) P_1(u^*) (1)$$ is a maximal partial isometry in $A$ with  $\|u\pm y_{\theta}\|= \sqrt{2}$, and $y_{\theta}\neq y$ for all $\theta$ in $\mathbb{R}\backslash \left(2 \pi \mathbb{Z}\cup \pi \frac{1+2 \mathbb{Z}}{2}\right)$ because $u$ is not unitary in $N$ nor in $A$. We further know from the just quoted proposition that $\|y-P_2(y)(y_{\theta})\|\leq 1-\cos(\theta) $, and hence $P_2(y) (y_{\theta})$ is invertible in $N_2(y)$ for $\theta$ close to zero. Since $y\in \partial_e(\mathcal{B}_{M})$, it follows from \cite[Lemma 2.2]{JamPerSidTah2015} that $y_{\theta}$ is Brown-Pedersen quasi-invertible in the terminology of \cite{JamPerSidTah2015}, which combined with the fact that $y_{\theta}$ is a tripotent in $N$ (and hence in $M$), trivially implies that $y_{\theta}\in \partial_e \left(\mathcal{B}_{M}\right)$. Therefore, for $\theta$ close to zero, $y_{\theta}\in \mathcal{M}_{u}\backslash\{y\}$ and $y_{\theta}\to y$ in norm when $\theta\to 0$, witnessing that $y$ is non-isolated in $\mathcal{M}_u$.
\end{proof}

Let us conclude this note with some afterthoughts on JB$^*$-triples. Let $E$ be a JB$^*$-triple with dimension at least 2. Suppose $u$ is a complete tripotent in $E$ which is not unitary. In view of Corollary \ref{c distance smaller than or equal to sqrt2} and Theorem \ref{t metric charact of unitaries JBalgebras}, a natural topic remains to be studied: Does the set $\mathcal{E}_{u} = \{e\in \partial_e(\mathcal{B}_E) : \|u\pm e\|\leq \sqrt{2} \}$ contains no isolated points?\smallskip

Every JB$^*$-triple $E$ admitting a unitary element is a unital JB$^*$-algebra with Jordan product and involution given in \eqref{eq product jordan}. Actually, there is a one-to-one (geometric) correspondence between the class of unital JB$^*$-algebras and the class of JB$^*$-triples admitting a unitary element. The next corollary is thus a rewording of our Theorem \ref{t metric charact of unitaries JBalgebras}.

\begin{corollary}\label{c metric charact of unitaries JBtriples} Let $E$ be a JB$^*$-triple admitting a unitary element. Suppose $u$ is an extreme point of the closed unit ball of $E$. Then the following statements are equivalent:
\begin{enumerate}[$(a)$]\item $u$ is a unitary tripotent;
\item The set $\mathcal{E}_{u} = \{e\in \partial_e(\mathcal{B}_E) : \|u\pm e\|\leq \sqrt{2} \}$ contains an isolated point.
\end{enumerate}
\end{corollary}

A typical example of a JB$^*$-triple admitting no unitary tripotents is a rectangular Cartan factor of type 1 of the form $C=B(H,K)$, of all bounded linear operators between two complex Hilbert spaces $H$ and $K$, with dim$(H)>$dim$(K)$.

In the simplest case $K=\mathbb{C}$ is one dimensional, and hence $C= H$ is a Hilbert space with triple product $\{ a,b,c \} = \frac12( \langle a, b\rangle c +  \langle c, b\rangle a)$ ($a,b,c\in H)$. Every norm-one element in $C$ is an extreme point of its closed unit ball, that is, $\partial_e(\mathcal{B}_{C}) = S(C)$. Let us fix $u\in S(C)$. By assuming dim$(C)\geq2$ it is not hard to see that $$\mathcal{C}_{u} = \{e\in \partial_e(\mathcal{B}_C) : \|u\pm e\|\leq \sqrt{2} \} = \{ i t u + x : t\in \mathbb{R}, x\in C, \langle e, x\rangle =0, t^2 + \|x\|^2 = 1\},$$ is pathwise-connected.\smallskip

In the case in which dim$(K)\geq 2$, every complete tripotent in $C$ must be a partial isometry $u$ satisfying $uu^* = \hbox{id}_{K}$ (and clearly, $u^* u \neq \hbox{id}_{H}$). Let us take $y \in \mathcal{C}_{u} = \{e\in \partial_e(\mathcal{B}_C) : \|u\pm e\|\leq \sqrt{2} \}$. We shall see that $y$ is non-isolated in $\mathcal{C}_{u}$. By Corollary \ref{c distance smaller than or equal to sqrt2} and Proposition \ref{p metric charact of unitaries JBtriples partial sufficient} we can assume that $y \in i \hbox{Symm}(C_2(u))$, that is, there exist two orthogonal tripotents $u_1,u_2$ with $u_1,u_2\leq u$, $u_1+u_2 = u$, and  $y = i (u_1-u_2)$. We may assume that $u_2\neq 0$. Let us take a minimal tripotent $e$ such that $e\leq u_2$, that is, $u_2 = (u_2-e) + e$ with $ (u_2-e) \perp e$. In this case $e = \xi\otimes \eta :\zeta\mapsto \langle\zeta, \eta\rangle \xi $ with $\eta\in S(H)$, $\xi\in S(K)$. Since $u^* u \neq \hbox{id}_{H}$, we can pick $\tilde\eta \in S(H)$ with $\langle \tilde\eta , u^*u (H)\rangle =\{0\}.$ The element $\tilde{e}= \xi\otimes \tilde\eta$ is a minimal tripotent in $C$ with $\tilde{e}\perp u_1, u_2-e$. It is not hard to check that, for each real $\theta$, the element $y_{\theta}:= i (u_1 - (u_2-e) - \cos(\theta) e +\sin(\theta) \tilde{e})$ is a complete tripotent in $C$, by orthogonality and from the fact that  $\| \alpha e + \beta \tilde{e}\|^2 = |\alpha|^2 + |\beta|^2$ for all $\alpha, \beta\in \mathbb{C}$, we can deduce that $$\| u \pm y_{\theta} \|  = \max\{ \|(1\pm i) u_1 \|, \|(1\mp i) (u_2-e) \|,  \|(1\pm i \cos (\theta)) e \pm \sin(\theta) \tilde{e}\| \}=\sqrt{2}.$$ Since $y\neq y_{\theta}\to y$ for $\theta\to 0$, we conclude that $y$ is non-isolated in $\mathcal{C}_{u}$ as claimed.

\medskip\medskip

\textbf{Acknowledgements} Authors partially supported by the Spanish Ministry of Science, Innovation and Universities (MICINN) and European Regional Development Fund project no. PGC2018-093332-B-I00 and Junta de Andaluc\'{\i}a grant FQM375. First author also supported by Universidad de Granada, Junta de Andaluc{\'i}a and Fondo Social Europeo de la Uni{\'o}n Europea (Iniciativa de Empleo Juvenil), grant number 6087.

\end{document}